\documentclass[a4paper,12pt]{article}

\usepackage{amsmath,amssymb,amsthm}
\usepackage{latexsym}
\usepackage{graphics}
\usepackage{graphicx,psfrag, cite}
\usepackage[bf, small]{titlesec}
\usepackage{authblk} %%% new line between authors
\usepackage{hyperref}
\usepackage{standalone}
\usepackage{pgfpages,tikz,pgfkeys,pgfplots}
\usetikzlibrary{arrows,positioning,matrix,fit,backgrounds,shapes,shapes.geometric, intersections}

\usetikzlibrary{shapes.callouts,decorations.pathmorphing} %%% 팝업용
\usetikzlibrary{shadows} % shadow
\usepackage{calc} % calculate angles to evenly space vertices in circular arrangements.
\usetikzlibrary{decorations.markings} %%put arrowheads in the middle of directed edges
\tikzstyle{vertex}=[circle, draw, inner sep=0pt, minimum size=6pt] % style
\newcommand{\vertex}{\node[vertex]}
\usetikzlibrary{arrows,matrix} %%% Hesse Diagram

\theoremstyle{plain}
\newtheorem{Thm}{Theorem}[section]

\newtheorem{Prop}[Thm]{Proposition}
\newtheorem{Cor}[Thm]{Corollary}

\theoremstyle{definition}
\newtheorem{Defi}[Thm]{Definition}
\newtheorem{Rem}[Thm]{Remark}

\newcommand{\CCC}{\mathcal{C}} %%%%%
\newcommand{\DDD}{\mathcal{D}} %%%%%
\title{Competition numbers of planar graphs
}

\author[1]{Jihoon Choi %\textsuperscript{$\ast$}
}
\author[2]{Soogang Eoh
\thanks{Corresponding author: mathfish@snu.ac.kr}}
\author[2]{Suh-Ryung Kim}

\affil[1]{Deparment of Mathematics Education, Cheongju University, Cheongju 28503, Republic of Korea}
\affil[2]{Department of Mathematics Education,
Seoul National University, Seoul 08826, Republic of Korea}

\date{}

\begin{document}

\maketitle

\begin{abstract}
In this paper, we relate the competition number of a graph to its edge clique cover number by presenting a tight inequality  $k(G) \ge \theta_e(G)-|V(G)|+\widetilde{k}(G)$ where $\theta_e(G)$, $k(G)$, and $\widetilde{k}(G)$ are the edge clique cover number, the competition number, and the co-competition number of a graph $G$, respectively.
By utilizing this inequality and a notion of  competition-effective edge clique cover, we obtain some meaningful results on competition numbers of planar graphs.
%
%
% we introduce the notion of co-competition number $\widetilde{k}(G)$ of a graph $G$ and show that  $\widetilde{k}(G)$ and the competition number $k(G)$ satisfy the inequality $k(G) \ge \theta_e(G)-|V(G)|+\widetilde{k}(G)$, which turns out to generalize the inequality $k(G) \ge \theta_e(G)-|V(G)|+2$ given by Opsut~\cite{opsut1982computation} as the co-competition number of a graph with at least one edge is at least two.
%In the process of deriving our inequality, we introduce a notion of competition-effective edge clique cover, which is possessed by a large family of graphs, and show that our inequality actually becomes an equality for a graph having a competition-effective edge clique cover and the competition number of a certain graph having a competition-effective edge clique cover may be obtained by utilizing its co-competition number.
%Especially, we present a precise relationship between the competition number and the co-competition number of a diamond-free plane graph.
\end{abstract}

%%%%%%%%%%%%%%%%%%%%%%%%%%%%%%%%%%%%%%%%%%%%%%%%%%%%%%%%%%%%%%%%%%%%%%%%%%%%%%%%%%%



%%%%%%%%%%%%%%%%%%%%%%%%%%%%%%%%%%%%%%%%%%%%%%%%%%%%%%%%%%%%%%%%%%%%%%%%%%%%%%%%%%%
%%%%%%%%%%%%%%%%%%%%%%%%%%%%%%%%%%%%%%%%%%%%%%%%%%%%%%%%%%%%%%%%%%%%%%%%%%%%%%%%%%%
\section{Introduction}
%%%%%%%%%%%%%%%%%%%%%%%%%%%%%%%%%%%%%%%%%%%%%%%%%%%%%%%%%%%%%%%%%%%%%%%%%%%%%%%%%%%

The \emph{competition graph} of a digraph $D$, denoted by $C(D)$, is defined as a graph which has the same vertex set as $D$ and has an edge $xy$ between two distinct vertices $x$ and $y$ if and only if, for some vertex $z \in V(D)$, the arcs $(x,z)$ and $(y,z)$ are in $D$
(see \cite{raychaudhuri1985generalized, roberts1999competition, lundgren1997p, roberts1983characterization, lu2009two} for reference).
The notion of competition graphs is due to Cohen~\cite{cohen1968interval} and arose from ecology.
Competition graphs also have applications in areas such as coding, radio
transmission, and modeling of complex economic systems.

Roberts~\cite{roberts1978food} observed that any graph $G$ together with $|E(G)|$ additional isolated vertices is the competition graph of an acyclic digraph.
Then he defined the \emph{competition number} of a graph $G$ to be the smallest number $k$ such that $G$ together with $k$ isolated vertices is the competition graph of an acyclic digraph, and denoted it by $k(G)$.

Computing the competition number of a graph is one of the important problems in the field of competition graphs.
Yet, computing the competition number of a graph is usually not easy as Opsut has shown that computation of the competition number in general is NP-hard in 1982.
While an upper bound $M$ of the competition number of a graph $G$ may be obtained by constructing an acyclic digraph whose competition graph is $G$ together with $M$ isolated vertices, getting a good lower bound is a very difficult task because there are usually so many cases to consider.
There has been much effort to compute competition numbers of graphs (for some results on competition numbers, see \cite{cho2005competition, li2009competition, kim2005graphs, kim2010competition, mckay2014competition, opsut1982computation, park2009competition, sano2013generalization, seager1990double, zhao2017competition, kim2017competition, factor20131, li2012competition, lee2010competition}).
In this paper, we seek for ways to compute competition numbers of planar graphs.

%Opsut~\cite{opsut1982computation} and Roberts~\cite{roberts1983characterization} used the notion of edge clique cover number of a graph to obtain the competition number.

In Section~2, we introduce the notion of competition-effective edge clique cover of a graph and give sufficient conditions for graphs having a competition-effective edge clique cover.
In Section~3, we introduce the notion of co-competition number of a graph $G$ denoted by $\widetilde{k}(G)$ and show that $k(G)$ and $\widetilde{k}(G)$ of a graph $G$ are related in terms of the edge clique cover number of $G$.
For a clique $K$ and an edge $e$ of a graph $G$, we say that $K$ \emph{covers} $e$ (or $e$ \emph{is covered by} $K$) if and only if $K$ contains the two end points of $e$.
An \emph{edge clique cover} of a graph $G$ is a collection of cliques that cover all the edges of $G$. The \emph{edge clique cover number} of a graph $G$, denoted by $\theta_e(G)$, is the smallest number of cliques in an edge clique cover of $G$.
Opsut~\cite{opsut1982computation} showed that $k(G) \ge \theta_e(G) - |V(G)| + 2$ for any graph $G$.
For a graph $G$, we could show that $k(G)$ and $\widetilde{k}(G)$ are related as  $k(G) \ge \theta_e(G)-|V(G)|+\widetilde{k}(G)$.
As a matter of fact, for a nonempty graph (a graph with at least one edge) $G$, $\widetilde{k}(G) \ge 2$ and our inequality generalizes the inequality $k(G) \ge \theta_e(G)-|V(G)|+2$ given by Opsut~\cite{opsut1982computation}.
Sano~\cite{sano2013generalization} gave a lower bound for the competition number which also generalizes Opsut's inequality but its viewpoint is different from ours.

For the graphs having a competition-effective edge clique cover such as nonempty diamond-free graphs, our inequality becomes an equality and the competition number may be computed in terms of the co-competition number.
Based on this observation, in Section~4, we give a sharp upper bound and a sharp lower bound for the competition number of a nontrivial connected diamond-free planar graph, each of which can be computed in a polynomial time.
%Moreover, we give the exact competition number of a nontrivial connected diamond-free planar graph satisfying a specific condition, which is also computed in a polynomial time.

%It is well-known that the competition number of a connected graph with at least one edge is greater than or equal to one, and the competition number of a chordal graph is less than or equal to one.

Every graph in this paper is assumed to be finite and simple unless otherwise stated.
For all undefined graph theoretical terms, see~\cite{bondy2008graph}.

%%%%%%%%%%%%%%%%%%%%%%%%%%%%%%%%%%%%%%%%%%%%%%%%%%%%%%%%%%%%%%%%%%%%%%%%%%%%%%%%%%%

%%%%%%%%%%%%%%%%%%%%%%%%%%%%%%%%%%%%%%%%%%%%%%%%%%%%%%%%%%%%%%%%%%%%%%%%%%%%%%%%%%%
\section{Competition-effective edge clique covers of graphs}
%%%%%%%%%%%%%%%%%%%%%%%%%%%%%%%%%%%%%%%%%%%%%%%%%%%%%%%%%%%%%%%%%%%%%%%%%%%%%%%%%%%

Given a graph $G$, let $\widetilde{\DDD}(G)$ be the set of acyclic digraphs the competition graph of each of which is $G$ together with $k(G)$ isolated vertices, that is,
\[
\widetilde{\DDD}(G) = \{ D \mid \text{$D$ is acyclic and $C(D)$ is $G$ together with $k(G)$ isolated vertices} \}.
\]

\noindent
Now we introduce the notion of competition-effective edge clique covers of graphs
\begin{Defi}\label{def:cover}
Let $G$ be a nonempty graph. % and competition number $k$.
A minimum edge clique cover $\CCC := \{C_1, \ldots, C_{\theta_e(G)}\}$ of $G$ is called an \emph{competition-effective edge clique cover} of $G$ if every clique in $\CCC$ is maximal in $G$ and there exists an acyclic digraph $D \in \widetilde{\DDD}(G)$ satisfying the following property.
%and the set $S_\CCC = \{w_1, \ldots, w_{\theta_e(G)}\}$ of vertices in $D$ with indegrees nonzero satisfying the following properties:
\begin{itemize}
%\item[(i)] The competition graph of $D$ is $G$ together with $k(G)$ isolated vertices.
\item[($\S$)] In $D$, there exist vertices $w_1, \ldots, w_{\theta_e(G)}$ such that $w_1, \ldots, w_{\theta_e(G)}$ are the only vertices of indegree nonzero in $D$ and $w_i$ is a common out-neighbor of all the vertices in $C_i$ for each $i=1,\ldots,\theta_e(G)$.
%In $D$, $w_i$ is a common out-neighbor of all the vertices in $C_i$ and is not an out-neighbor of any vertex not in $C_i$ for each $i=1,2,\ldots,\theta_e(G)$.
%there exists a vertex $w_i$ in $D$ such that $w_i$ is a common out-neighbor of all the vertices in $C_i$.
%Moreover, $w_1, \ldots, w_{\theta_e(G)}$ are the only vertices of indegree nonzero in $D$.
%\item[(iii)] $\CCC = \{ N^-_D(v) \mid v \in V(D), d^-_D(v) \ge 1 \}$.
\end{itemize}
For a competition-effective edge clique cover $\CCC$ and a digraph $D$ in Definition~\ref{def:cover}, we say that $D$ is a \emph{digraph accompanying $\CCC$}.
By definition, $w_i$ in the property ($\S$) is not an out-neighbor of any vertex not in $C_i$ for each $i=1,\ldots,\theta_e(G)$.
In this vein, we call $w_i$ in the property ($\S$) of Definition~\ref{def:cover} the \emph{sink of $C_i$} in $D$ for each $i=1,\ldots,\theta_e(G)$.
%By the definition of competition-effective edge clique cover, among the digraphs accompanying a competition-effective edge clique cover $\CCC$ of a graph, there exists a digraph $D$ in which the vertices in

The graph $G$ given in Figure~\ref{fig:effective} has a competition-effective edge clique cover.
The competition number of $G$ is one and the competition graph of $D$ in Figure~\ref{fig:effective} is $G \cup \{v_0\}$.
Now consider the family $\CCC = \{ C_1=\{v_8,v_9\}, C_2=\{v_7,v_8\}, C_3=\{v_6,v_7\}, C_4=\{v_4,v_5,v_6,v_9\}, C_5=\{v_3,v_4,v_5\}, C_6=\{v_1,v_2,v_4,v_9\} \}$ of maximal cliques of $G$.
It can easily be checked that $\CCC$ is a minimum edge clique cover of $G$.
Moreover, the $i$th term of $(v_7,v_6,v_5,v_3,v_2,v_0)$ is a common out-neighbor of all the vertices in $C_i$ for $i=1,\ldots,6$ and those terms are the only vertices of indegree nonzero in $D$.
%for each $i=1,2,\ldots,\theta_e(G)$, we call the vertex denoted by $w_i$ %($i=1,2,\ldots,\theta_e(G)$) in (ii) the \emph{sink} of $C_i$ in $D$.
\end{Defi}

\begin{figure}
\begin{center}
\begin{tikzpicture}[x=1.5cm, y=1.5cm]

    \vertex (b1) at (1,3) [label=left:$v_7$]{};
    \vertex (b2) at (1,2) [label=left:$v_8$]{};
    \vertex (b3) at (2,3) [label=above:$v_6$]{};
    \vertex (b4) at (2,2) [label=below left:$v_9$]{};
    \vertex (b5) at (2,1) [label=below:$v_2$]{};
    \vertex (b6) at (3,3) [label=above right:$v_5$]{};
    \vertex (b7) at (3,2) [label=below right:$v_4$]{};
    \vertex (b8) at (3,1) [label=below:$v_1$]{};
    \vertex (b9) at (4,2.5) [label=right:$v_3$]{};
   % \vertex (b7) at (2,0) [label=below:$v_7$]{};
%    \vertex (b8) at (1.5, -0.5) [label=below:$D$]{};

    \path
 (b1) edge [-,green,thick] (b2)
 (b2) edge [-,red,thick] (b4)
 (b3) edge [-,orange,thick] (b4)
 (b1) edge [-,blue,thick] (b3)

 (b4) edge [-,orange,thick] (b6)
 (b6) edge [-,orange,thick] (b7)
 (b7) edge [-,orange,thick] (b3)
 (b4) edge [-,orange,thick] (b7)
 (b3) edge [-,orange,thick] (b6)

 (b4) edge [-,purple,thick] (b5)
 (b5) edge [-,purple,thick] (b7)
 (b7) edge [-,purple,thick] (b8)
 (b8) edge [-,purple,thick] (b4)
 (b5) edge [-,purple,thick] (b8)
 (b4) edge [-,purple,thick] (b7)

 (b6) edge [-,thick] (b7)
 (b6) edge [-,thick] (b9)
 (b7) edge [-,thick] (b9)

;
 \draw (2.5,0.3) node{$G$};
\end{tikzpicture}
\qquad \qquad
\begin{tikzpicture}[x=1.5cm, y=1.5cm]

 \vertex (b1) at (1,3) [label=left:$v_7$]{};
 \vertex (b2) at (1,2) [label=left:$v_8$]{};
 \vertex (b3) at (2,3) [label=below right:$v_6$]{};   \vertex (b4) at (2,2) [label=below left:$v_9$]{};
 \vertex (b5) at (2,1) [label=below:$v_2$]{};
 \vertex (b6) at (3,3) [label=above right:$v_5$]{};
 \vertex (b7) at (3,2) [label=left:$v_4$]{};
 \vertex (b8) at (3,1) [label=below:$v_1$]{};
 \vertex (b9) at (4,2.5) [label=right:$v_3$]{};
 \vertex (b10) at (4,1.5) [label=right:$v_0$]{};
    \path
 (b1) edge [->,green,thick] (b3)
 (b2) edge [->,green,thick] (b3)

 (b2) edge [->,red,thick] (b1)
 (b4) edge [->,red,thick] (b1)

 (b1) edge [->,blue,bend left=25,thick] (b6)
 (b3) edge [->,blue,thick] (b6)

 (b3) edge [->,orange,thick] (b9)
 (b4) edge [->,orange,thick] (b9)
 (b6) edge [->,orange,thick] (b9)
 (b7) edge [->,orange,thick] (b9)

 (b6) edge [->,thick] (b5)
 (b7) edge [->,thick] (b5)
 (b9) edge [->,thick] (b5)

 (b4) edge [->,purple,thick] (b10)
 (b5) edge [->,purple,thick] (b10)
 (b7) edge [->,purple,thick] (b10)
 (b8) edge [->,purple,thick] (b10)

 ;

 \draw (2.5,0.3) node{$D$}
	;
\end{tikzpicture}

\end{center}
\caption{A graph $G$ having a competition-effective edge clique cover and an acyclic digraph $D$ accompanying the competition-effective edge clique cover}
\label{fig:effective}
\end{figure}

%
%\begin{figure}
%\psfrag{a}{$v_9$}
%\psfrag{b}{$v_8$}
%\psfrag{c}{$v_7$}
%\psfrag{d}{$v_6$}
%\psfrag{e}{$v_5$}
%\psfrag{f}{$v_4$}
%\psfrag{g}{$v_3$}
%\psfrag{h}{$v_2$}
%\psfrag{i}{$v_1$}
%\psfrag{j}{$v_0$}
%\psfrag{k}{$G$}
%\psfrag{l}{$D$}
%\begin{center}
%\includegraphics[width=6cm]{effective1.eps} \qquad
%\includegraphics[width=6cm]{effective2.eps}
%\end{center}
%\caption{A graph $G$ having a competition-effective edge clique cover and an acyclic digraph $D$ accompanying the competition-effective edge clique cover}
%\label{fig:effective}
%\end{figure}

%\begin{Rem}\label{rmk:maximal}
Since a competition-effective edge clique cover $\CCC$ of a graph is a minimum edge clique cover, the sinks of cliques belonging to $\CCC$ in a digraph accompanying $\CCC$ are all distinct.
%\end{Rem}

%\begin{Rem}\label{rmk:maximal}
%Given a graph $G$ and a minimum edge clique cover $\CCC$, we may expand each clique in $\CCC$ to a maximal clique.
%Therefore any graph has a minimum edge clique cover consisting of maximal cliques.
%it is sufficient to consider edge clique covers consisting of only maximal cliques when we study a competition-effective edge clique cover of a graph.
%\end{Rem}

%\begin{Rem}\label{rmk:isolated}
%Having a competition-effective edge clique cover is not affected by whether or not a graph has isolated vertices.
%In this context, we assume that every graph dealt in this section has no isolated vertex unless otherwise stated.
%\end{Rem}

A \emph{perfect elimination ordering} of a graph $G$ with $N$ vertices is an ordering $[v_1, v_2, \ldots, v_n]$ of the vertices of $G$ such that the neighborhood of $v_i$ is a clique in $G_i := G[v_{i}, v_{i+1} \ldots, v_n]$ for each $i=1,\ldots,n-1$.
It is well-known that every chordal graph has a perfect elimination ordering.

It is also well known that a digraph $D$ is acyclic if and only if there exists a bijection $\ell: V(D) \to \{1,2,\ldots,|V(D)|\}$ such that whenever there is an arc from a vertex $u$ to a vertex $v$, $\ell(u) > \ell(v)$. We call such a function $\ell$ an \emph{acyclic labeling} of $D$.

Given a graph $G$ and a minimum edge clique cover $\CCC$, we may expand each clique in $\CCC$ to a maximal clique.
Therefore any graph has a minimum edge clique cover consisting of maximal cliques.

Given a graph $G$ and a vertex $v$ of $G$, we denote by $N_G[v]$ (resp.\ $N_G(v)$) the closed neighborhood (resp.\ open neighborhood) of $v$ in $G$.

Now we present some sufficient conditions under which a graph has a competition-effective edge clique cover.
\begin{Thm}\label{thm:chordal}
Every nonempty chordal graph has a competition-effective edge clique cover.
\end{Thm}
\begin{proof}
Take a nonempty chordal graph $G$.
If a graph $H$ has a competition-effective edge clique cover without having isolated vertices, then $H$ together with isolated vertices still has a competition-effective edge clique cover.
In this context, we may assume that $G$ has no isolated vertices.
%Since the existence of isolated vertices do not affect whether a graph has a competition-effective edge clique cover or not, we may assume that $G$ has no isolated vertex.
Since $G$ is chordal, there exists a perfect elimination ordering $[v_1, v_2, \ldots, v_n]$ of $G$.
Let $G_i = G[v_i, v_{i+1}, \ldots, v_n]$, and $N_i = N_{G_i}[v_i]$ for each $i = 1, 2, \ldots, n-1$.
In addition, let $\theta = \theta_e(G)$ and $\CCC = \{C_1, \ldots, C_\theta\}$ be a minimum edge clique cover of $G$ consisting of maximal cliques.

Given a subset $X = \{v_{r_1}, v_{r_2}, \ldots, v_{r_j}\}$ of $V(G)$ with ${r_1} < {r_2} < \cdots < r_j$ for some positive integer $j$, we may correspond the ordered $j$-tuple $\mathbf{a}(X):= (r_1, r_2, \ldots, r_j)$.
We rearrange $C_1, \ldots, C_\theta$ so that $\mathbf{a}(C_1) \prec \mathbf{a}(C_2) \cdots \prec \mathbf{a}(C_\theta)$ where $\prec$ is the lexicographic order.

%Fix $k \in \{1,\ldots,\theta\}$. Let $n_k = \min \{ j \mid v_j \in C_i \text{ for some } i \in \{k,k+1,\ldots,\theta\} \}$.
%That is, $v_{n_k}$ is the vertex with smallest index in $\bigcup_{i=k}^\theta C_i$.
To show that $\CCC$ consists of some elements in $\{N_1, N_2, \ldots, N_{n-1}\}$, fix $l \in \{1,\ldots,\theta\}$ and let $n_l = \min \{ j \mid v_j \in C_l \}$.
%We will show that $N_{n_k} = C_k$. %in the following.
By the choice of ${n_l}$, %$\bigcup_{i=k}^\theta C_i \subset \{v_{n_k}, v_{n_k+1}, \ldots, v_n \}$, so $C_k \subset \bigcup_{i=k}^\theta C_i \subset V(G_{n_k})$.
\begin{equation}\label{eqn:chordal}
C_l \subset V(G_{n_l}).
\end{equation}
%On the other hand, $v_{n_k} \in C_k$ by the arrangement of $C_1, \ldots, C_\theta$.
Thus $C_l$ is a clique in $G_{n_l}$ containing $v_{n_l}$.
By the definition of perfect elimination ordering, $v_{n_l}$ is a simplicial vertex of $G_{n_l}$ and so $C_{l} \subset N_{n_l}$.
However, since $C_l$ is a maximal clique in $G$, it is also a maximal clique in $G_{n_l}$.
Therefore $C_l = N_{n_l}$ as $N_{n_l}$ is a clique in $G_{n_l}$ and so $\CCC$ consists of some elements in $\{N_1, N_2, \ldots, N_{n-1}\}$.
Thus $\CCC = \{N_{n_1}, \ldots, N_{n_\theta}\}$.
Note that $n_1=1$.
To see why, recall our assumption that $G$ has no isolated vertices.
Since $C_1, \ldots, C_\theta$ are arranged based upon the lexicographic order of $\mathbf{a}(C_1), \ldots, \mathbf{a}(C_\theta)$, $v_1$ belongs to $C_1$.

%Now we construct a digraph $D$ whose competition graph is $G$ together with an isolated vertex.
Let $D$ be a digraph with $V(D) = V(G) \cup \{v_0\}$ and
\[A(D) = \bigcup_{l=1}^\theta \{(v,v_{n_{l}-1}) \mid v \in N_{n_l}\}.
\]
Now the following are true:
\begin{align*}
v_i v_j \in E(G)
&\Leftrightarrow v_iv_j \in N_{n_l} \text{ for some } l \in \{1,\ldots,\theta \} \\
&\Leftrightarrow (v_i, v_{n_{l}-1}) \in A(D) \text{ and } (v_j, v_{n_{l}-1}) \in A(D) \text{ for some } l \in \{1,\ldots,\theta \} \\
&\Leftrightarrow v_i v_j \in E(C(D)).
\end{align*}
Thus the competition graph of $D$ is $G$ together with the isolated vertex $v_0$.
To show that $D$ is acyclic, take $(v_i,v_j) \in A(D)$.
Then $v_i \in N_{n_l}$ and $j=n_l-1$ for some $l \in \{1,\ldots,\theta\}$.
Now, since $N_{n_l} = C_l$, by (\ref{eqn:chordal}),
$N_{n_l} \subset V(G_{n_l}) = \{v_{n_l}, v_{{n_l}+1}, \ldots, v_n\}$, which implies $i \ge n_l$.
Thus $i > j$ and so $D$ is acyclic.
By the definition of $D$, $v_{n_1-1}, \ldots, v_{n_\theta-1}$ are the only vertices of indegree nonzero in $D$ and $\CCC$ satisfies ($\S$) of Definition~\ref{def:cover}.
This completes the proof.
\end{proof}

\begin{Rem}
In the proof given above, %of Theorem~\ref{thm:chordal}
we have actually shown a stronger statement that every minimum edge clique cover of a chordal graph consisting of maximal cliques is a competition-effective edge clique cover.
\end{Rem}

Given a maximal clique $C$ of a graph $G$, an edge of $G$ is said to be \emph{occupied by $C$} if $C$ is the only maximal clique that covers it.

\begin{Thm}\label{thm:diamond}
Suppose that a nonempty graph $G$ satisfies the property that, for each maximal clique $C$, each vertex of $C$ is an end vertex of an edge occupied by $C$.
Then $G$ has a competition-effective edge clique cover.
\end{Thm}
\begin{proof}
Let $\CCC$ be the set of all maximal cliques of $G$.
We first show that $\CCC$ is a minimum edge clique cover.
Obviously $\CCC$ is an edge clique cover of $G$.
Let $\CCC^*$ be a minimum edge clique cover of $G$ consisting of maximal cliques.
Then clearly $\CCC^* \subset \CCC$.
By the hypothesis, every clique in $\CCC$ has an edge occupied by it.
Since an edge occupied by a maximal clique cannot be covered by any other maximal cliques, $\CCC$ with an element omitted no longer covers the edges of $G$.
Therefore $\CCC^* = \CCC$ and so $\CCC$ is a minimum edge clique cover.

Let $D$ be a digraph in $\widetilde{\DDD}(G)$ and $\ell$ be an acyclic labeling of $D$.% where $n=|V(G)| + k(G)$.

Take a maximal clique $C$ in $\CCC$.
Then the end vertices of each of the edges occupied by $C$ has a common out-neighbor in $D$.
We consider such common out-neighbors and take one, say $x_C$, with the smallest $\ell$-value among them.
Since $x_C$ is a common out-neighbor of the end vertices of an edge occupied by $C$ in $D$, $x_C \neq x_{C'}$ if $C \neq C'$ for $C, C' \in \CCC$.
We consider the digraph $D^*$ with the vertex set $V(D)$ and the arc set
\[
\bigcup_{C \in \CCC} \{(v, x_C) \mid v \in C \}.
\]
Since $x_C \neq x_{C'}$ for distinct $C, C' \in \CCC$, the competition graph of $D^*$ is $G$ together with $k(G)$ isolated vertices.
Therefore it remains to show that $D^*$ is acyclic in order to prove $D^* \in \widetilde{D}(G)$.

Take an arc in $D^*$.
Then it is in the form of $(v,x_C)$ for some $C \in \CCC$ and $v \in C$.
By the hypothesis, $v$ is an end vertex of an edge occupied by $C$.
Then $v$ and the other end of that edge have a common out-neighbor $x$ in $D$.
Since $\ell$ is an acyclic labeling of $D$, $\ell(x) < \ell(v)$.
By the choice of $x_C$, $\ell(x_C) \le \ell(x)$, so $\ell(x_C) < \ell(v)$.
Thus $\ell$ is still an acyclic labeling of $D^*$ and so $D^*$ is acyclic.
Hence
$D^* \in \widetilde{\DDD}(G)$.
By the definition of $D^*$, it is obvious that $\CCC$ satisfies ($\S$) of Definition~\ref{def:cover} and the theorem statement holds.
\end{proof}

A \emph{diamond} is a graph obtained from $K_4$ by deleting an edge.
%By a \emph{diamond} of a graph, we mean a diamond as an induced subgraph.
A graph is called \emph{diamond-free} if it does not contain a diamond as an induced subgraph.
It is easy to see that a graph is diamond-free if and only if no two of its maximal cliques cover a common edge.

\begin{Cor}\label{cor:diamond2}
%Every graph $G$ with at least one edge such that any two diamonds of $G$ are edge-disjoint has a competition-effective edge clique cover. In particular,
Every nonempty diamond-free graph has a competition-effective edge clique cover.
\end{Cor}
\begin{proof}
If $G$ is a nonempty diamond-free graph, then each edge of $G$ is occupied by a maximal clique.
Hence the corollary immediately follows from Theorem~\ref{thm:diamond}.
\end{proof}
\section{A new parameter of a graph related to competition number}
%%%%%%%%%%%%%%%%%%%%%%%%%%%%%%%%%%%%%%%%%%%%%%%%%%%%%%%%%%%%%%%%%%%%%%%%%%%%%%%%%%%

In this section, we first introduce the notion of the co-competition number of a graph.
Then we give a new lower bound for the competition number of a graph in terms of its co-competition number.
In addition, we show that the graphs having competition-effective edge clique covers have the lower bound as their competition numbers.

\begin{Defi}
Let $G$ be a graph.
Among the numbers of vertices of indegree $0$ in the digraphs in $\widetilde{\DDD}(G)$, we call the maximum the \emph{co-competition number} of $G$ and denote it by $\widetilde{k}(G)$, that is,
\[
\widetilde{k}(G) = \max \{ i(D) \mid \text{$i(D)$ denotes the number of vertices of indegree $0$ in $D \in \widetilde{\DDD}(G)$} \}.
\]
We say that $\widetilde{k}(G)$ is {\it attained by} $D$ if $\widetilde{k}(G) = i(D)$.
\end{Defi}
\noindent For a graph $G$, the number of vertices with indegree $0$ in a digraph belonging to $\widetilde{\DDD}(G)$ is less than or equal to $|V(G)|$. Thus $\widetilde{k}(G)$ is finite.

%\begin{Defi}
%Given a graph $G$, we consider the digraphs the competition graph of each of which is $G$ together with $k(G)$ isolated vertices.
%Then we count the number of vertices of indegree $0$ in each of those digraphs. We call  the maximum number among these numbers of vertices of indegree $0$ the \emph{co-competition number} of $G$ and denote it by $\widetilde{k}(G)$.
%\end{Defi}

\begin{Prop}\label{prop:base2}
For any graph $G$ with at least two vertices, $\widetilde{k}(G) \ge 2$ \end{Prop}
\begin{proof}
Let $G$ be a graph with at least two vertices.
%Then $G$ together with $k(G)$ isolated vertices is the competition graph of some acyclic digraph with $n$ vertices for some positive integer $n \ge 2$.
In $\widetilde{D}(G)$, we take a digraph $D$ with arcs as few as possible.
Since $G$ has at least two vertices, $n := |V(D)| \ge 2$.
Let $\ell$ be an acyclic labeling of $D$.
We denote by $v_i$ the vertex $v$ satisfying $\ell(v)=i$.
%where $V(D)=\{v_1, \ldots, v_n\}$.
By definition of acyclic labeling, $v_n$ is of indegree $0$.
Suppose that the indegree of $v_{n-1}$ is nonzero.
By definition of acyclic labeling, $v_n$ is the only in-neighbor of $v_{n-1}$, which implies that $v_{n-1}$ does not induce any edge in $G$ as a common out-neighbor of two vertices.
%$v_{n-1}$ is not a common out-neighbor of any two vertices.
Therefore we may delete the arc $(v_n, v_{n-1})$ to obtain an acyclic digraph in $\widetilde{\DDD}(G)$, which contradicts the choice of $D$.
%Then the digraph $D^*$ obtained from $D$ by deleting the arc $(v_n, v_{n-1})$ has less arcs than $D$ and $C(D^*)=C(D)$, which contradicts the choice of $D$.
Thus $v_{n-1}$ is of indegree $0$.
Hence $\widetilde{k}(G) \ge 2$.
\end{proof}

\begin{Thm}\label{prop:basenumber2}
Let $G$ be a graph with a competition-effective edge clique cover $\CCC$. %$\{C_1, \ldots, C_{\theta_e(G)}\}$ and a digraph $D$ accompanying $\CCC$.
Then the cliques in $\CCC$ can be labeled as $C_1, \ldots, C_{\theta_e(G)}$ so that
%$\widetilde{k}(G) \ge \left| \cup_{i=1}^k C_i \setminus \{w_1,w_2, \ldots, w_{k-1}\} \right|$ for any $1 \le k \le \theta_e(G)$ where $w_i$ is the sink of $C_i$ in $D$ for each $i=1,\ldots,\theta_e(G)$. Consequently,
$\widetilde{k}(G) \ge \left| \bigcup_{i=k}^{\theta_e(G)} C_i \right| - \theta_e(G) + k$ for any $1 \le k \le \theta_e(G)$.
\end{Thm}
\begin{proof}
Let $D$ be a digraph accompanying $\CCC$.
We consider the subdigraph $D'$ of $D$ induced by the sinks of the cliques belonging to $\CCC$ in $D$.
Obviously $D'$ is acyclic, so we may label the vertices of $D'$ as $w_1, \ldots, w_{\theta_e(G)}$ so that
\[
\begin{minipage}{0.85\textwidth}
\begin{center}
$(w_j,w_i)$ is an arc only if $i < j$.
\end{center}
\end{minipage}
\tag{$\star$}
\]
Now we label the cliques in $\CCC$ as $C_1, \ldots, C_{\theta_e(G)}$ so that $w_i$ is the sink of $C_i$ for each $i=1,\ldots,\theta_e(G)$.

Fix $1 \le k \le \theta_e(G)$ and suppose that $\bigcup_{i=k}^{\theta_e(G)} C_i \setminus \{w_{k+1}, \ldots, w_{\theta_e(G)}\}$ contains a sink $w_j$ for some $j \in \{1,\ldots,k\}$.
Then $w_j \in C_i$ for some $i \in \{k, k+1, \ldots, \theta_e(G)\}$.
Since $w_i$ is the sink of $C_i$ and $w_j \in C_i$, there is an arc from $w_j$ to $w_i$ and so, by $(\star)$, $i<j$, which is a contradiction.
Therefore $\bigcup_{i=k}^{\theta_e(G)} C_i \setminus \{w_{k+1}, \ldots, w_{\theta_e(G)}\}$ does not contain a sink, that is, every vertex in $\bigcup_{i=k}^{\theta_e(G)} C_i \setminus \{w_{k+1}, \ldots, w_{\theta_e(G)}\}$ has indegree $0$ in $D$.
By the definition of $\widetilde{k}(G)$,
\[
\widetilde{k}(G) \ge \left| \bigcup_{i=k}^{\theta_e(G)} C_i \setminus \{w_{k+1}, \ldots, w_{\theta_e(G)}\} \right|  \ge \left| \bigcup_{i=k}^{\theta_e(G)} C_i \right| - \theta_e(G) + k.
\]
and the theorem follows.
\end{proof}

Opsut~\cite{opsut1982computation} showed that, for any graph $G$, $k(G) \ge \theta_e(G) - |V(G)| + 2$.
We generalize this inequality. To do so, we need the following theorem.

\begin{Thm}\label{thm:hall}
Let $G$ be a nonempty graph and $D$ be a digraph in $\widetilde{\DDD}(G)$.
Then $D$ has at least $\theta_e(G)$ vertices  of indegree nonzero.
\end{Thm}
\begin{proof}
%Let $D$ be an acyclic digraph whose competition graph is $G$ together with $k(G)$ isolated vertices and
Let $\CCC$ be a minimum edge clique cover of $G$ consisting of maximal cliques $C_1$, $\ldots$, $C_{\theta_e(G)}$.  We define
\[
A_i := \{ v \in V(D) \mid \text{ $v$ is a common out-neighbor of two vertices in $C_i$}\}
\]
for each $i=1,\ldots,\theta_e(G)$.
Since $\CCC$ is an edge clique cover of $G$, $G[C_i]$ contains at least one edge whose two end vertices, therefore, has a common out-neighbor in $D$ and so $A_i \neq \emptyset$ for each $i=1,\ldots,\theta_e(G)$.

Let $B = (X,Y)$ be a bipartite graph, where $X = \{A_1, A_2, \ldots, A_{\theta_e(G)}\}$ and $Y = V(D)$, such that, for $A_i \in X$ and $v \in Y$, $\{A_i, v\}$ is an edge of $B$ if and only if $v \in A_i$.
By definition, $A_{i} = N_B(A_{i})$ for each $i = 1,\ldots,\theta_e(G)$.
To show that $B$ satisfies Hall's condition for Hall's marriage theorem, suppose, to the contrary, that there exists $S \subset X$ such that $|S| > |N_B(S)|$.
We denote $S = \{A_{i_1}, A_{i_2}, \ldots, A_{i_k} \}$ and $N_B(S) = \{z_1,z_2,\ldots,z_l\}$. Then $k > l$ by our assumption.

%By the construction of $B$ and the definition of $A_i$,
%that is contained in $\bigcup_{j=1}^k A_{i_j}$.
To show that $N^-(z_1), N^-(z_2), \ldots, N^-(z_l)$ in $D$ cover all the edges covered by $C_{i_1}$, $C_{i_2}$, $\ldots$, $C_{i_k}$, take an edge $e$ of $G$ covered by $C_{i_j}$ for some $j \in \{1,\ldots,k\}$.
Then there exists a vertex $z \in A_{i_j}$ such that $z$ is a common out-neighbor of the end vertices of $e$.
Therefore $e$ is covered by $N^-(z)$.
Since $A_{i_j} \in X$, $A_{i_j} = N_B(A_{i_j})$.
Since $A_{i_j} \in S$, $N_B(A_{i_j}) \subset N_B(S)$.
Therefore $z \in N_B(S)$ since $z \in A_{i_j}$ and $A_{i_j} = N_B(A_{i_j})$.
%Thus $e$ is covered by $N^-(z)$.
As the in-neighborhood of each vertex of $N_B(S)$ clearly forms a clique in $G$, $\{N^-(z_1), N^-(z_2), \ldots, N^-(z_l) \}$ is a family of cliques of $G$ covering all the edges covered by $\{C_{i_1}, C_{i_2}, \ldots, C_{i_k} \}$.
Since $k > l$, we replace $C_{i_1}, C_{i_2}, \ldots, C_{i_k}$ with $N^-(z_1), N^-(z_2), \ldots, N^-(z_l)$ in $\CCC$ to obtain a new edge clique cover of $G$ consisting of fewer cliques than $\CCC$, a contradiction. Thus $B$ satisfies Hall's condition and so, by Hall's marriage theorem, $B$ has a matching $M = \{ \{A_i, w_i\} \mid i = 1,\ldots,\theta_e(G) \}$ that saturates $X$.
By the definition of $A_i$, $w_1, w_2, \ldots, w_{\theta_e(G)}$ are vertices of indegree at least two.
Since each of $w_1, w_2, \ldots, w_{\theta_e(G)}$ is saturated by the matching $M$, $w_1, w_2, \ldots, w_{\theta_e(G)}$ are all distinct and the theorem statement follows.
\end{proof}

\begin{Thm}\label{prop:cn}
For any graph $G$, $k(G) \ge \theta_e(G) - |V(G)| + \widetilde{k}(G)$.
\end{Thm}

\begin{proof}
If $G$ is an empty graph, then obviously $k(G)=0$, $\theta_e(G)=0$, $\widetilde{k}(G)=|V(G)|$, and the inequality holds.
Let $G$ be a nonempty graph and $D$ be a digraph by which $\widetilde{k}(G)$ is attained. By definition,
\begin{equation}\label{eqn:cn1}
|V(D)|=|V(G)|+k(G).
\end{equation}
On the other hand, by Theorem~\ref{thm:hall}, there are at least $\theta_e(G)$ vertices of indegree nonzero in $D$.
Since any vertex in $D$ has indegree $0$ or indegree nonzero,
%Since $D$ has exactly $\widetilde{k}(G)$ vertices of indegree $0$,
\begin{equation}\label{eqn:cn2}
|V(D)| \ge \widetilde{k}(G)+ \theta_e(G).
\end{equation}
Then (\ref{eqn:cn1}) and (\ref{eqn:cn2}) yield the desired inequality.
\end{proof}

\noindent Opsut's inequality for competition numbers immediately follows from Proposition~\ref{prop:base2} and Theorem~\ref{prop:cn}.

\begin{Prop}\label{prop:basenumber1}
Let $G$ be a graph with a competition-effective edge clique cover $\CCC$.
Then $\widetilde{k}(G)$ is attained by any digraph accompanying $\CCC$. Furthermore,
\begin{align*}
|V(D)| = \widetilde{k}(G)+ \theta_e(G)
\end{align*}
for any digraph $D$ accompanying $\CCC$.
\end{Prop}
\begin{proof}
Let $D$ be a digraph accompanying $\CCC$.
By definition, $D$ has exactly $\theta_e(G)$ vertices of indegree nonzero.
Therefore, by Theorem~\ref{thm:hall}, $D$ has the smallest number of vertices of indegree nonzero. This implies that $D$ has the largest number of vertices of indegree $0$.
By the definition of $\widetilde{k}(G)$, $\widetilde{k}(G)$ is attained by $D$.
Now, since the number of vertices of indegree nonzero is $\theta_e(G)$ and the number of vertices of indegree $0$ is $\widetilde{k}(G)$ in $D$, we have $|V(D)| = \widetilde{k}(G)+ \theta_e(G)$.
\end{proof}

In the following, we present some sufficient conditions for a graph to make the inequality given in  Theorem~\ref{prop:cn} sharp.

\begin{Prop}\label{prop:cne}
Let $G$ be a graph with a competition-effective edge clique cover  $\CCC$. Then $k(G) = \theta_e(G) - |V(G)| + \widetilde{k}(G)$.
\end{Prop}
\begin{proof}
Let $D$ be a digraph accompanying $\CCC$.
By Proposition~\ref{prop:basenumber1},
\begin{align*}
|V(D)| = \widetilde{k}(G)+ \theta_e(G).
\end{align*}
By the definition of competition number,
\begin{align*}
|V(D)| = |V(G)|+k(G).
\end{align*}
Then the above equalities yield the desired equality.
\end{proof}

\begin{figure}
\begin{center}
\includegraphics[width=4cm]{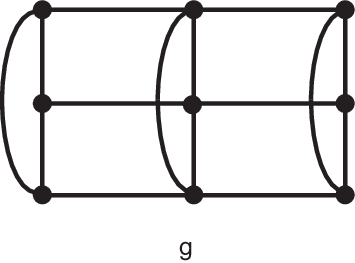}
\end{center}
\vskip-0.5cm
\caption{A graph $G$ with $\widetilde{k}(G) = 3$}
\label{fig:basenumber}
\end{figure}

Now we present an example showing how Theorem~\ref{prop:basenumber2} and Proposition~\ref{prop:cne} can be applied to obtain a lower bound for the competition number of a graph having a competition-effective edge clique cover.
The graph $G$ given in Figure~\ref{fig:basenumber} is a nonempty diamond-free graph, so it has a competition-effective edge clique cover by Corollary~\ref{cor:diamond2}.
Note that $\theta_e(G)=9$ and the union of any three maximal cliques of $G$ contains at least five vertices.
By applying Theorem~\ref{prop:basenumber2} for $k=7$, we have $\widetilde{k}(G) \ge 3$.
Thus, by Proposition~\ref{prop:cne}, $k(G) = \theta_e(G) - |V(G)| + \widetilde{k}(G) \ge 3$.
This bound is sharper than Opsut's bound in~\cite{opsut1982computation} that $k(G) \ge \theta_e(G) - |V(G)| + 2 = 2$.
In fact, we can show $k(G) \le 3$ by constructing an acyclic digraph whose competition graph is $G$ together with three isolated vertices, and therefore $k(G) = 3$.
Since $k(G) = \theta_e(G) - |V(G)| + \widetilde{k}(G)$, we have $\widetilde{k}(G)=3$.

The following proposition guarantees that the equality $k(G) = \theta_e(G) - |V(G)| + \widetilde{k}(G)$ given in Proposition~\ref{prop:cne} is still true under the following condition.

\begin{Prop}\label{prop:formula}
Let $G$ be a nonempty graph and $\CCC$ be a set of maximal cliques of $G$.
Suppose that every clique in $\CCC$ has an edge that is occupied by it.
Then $k(G) = \theta_e(G) - |V(G)| + \widetilde{k}(G)$.
\end{Prop}
\begin{proof}
Since $\CCC$ is a set of maximal cliques of $G$, $\CCC$ is an edge clique cover of $G$.
Moreover, by the hypothesis that every clique in $\CCC$ has an edge that is occupied by it, it is a minimum edge clique cover.

Let $D$ be a digraph with the most arcs among which acyclic digraphs  by which $\widetilde{k}(G)$ are attained.
We show that the number $r$ of vertices with at least one in-neighbor in $D$ equals $\theta_e(G)$.
By Theorem~\ref{thm:hall}, $r \ge \theta_e(G)$.
To reach a contradiction, we assume $r > \theta_e(G)$.
Let $w_1, \ldots, w_r$ be the vertices with at least one in-neighbor in $D$.
Since the co-competition number of $G$ is attained by $D$, $w_i$ has at least two in-neighbors in $D$ and so $N^-_D(w_i)$ forms a clique of size at least two for each $i=1,\ldots,r$.
Let $C_i$ be a maximal clique including $N^-_D(w_i)$ for each $i=1,\ldots,r$.
Then $C_i$ belongs to $\CCC$.
Since $|\CCC|=\theta_e(G)$ and $r > \theta_e(G)$, $C_p = C_q$ for some distinct $p,q \in \{1,\ldots,r\}$ by the Pigeonhole principle.
Without loss of generality, we may assume $w_p$ has a lower label than $w_q$ in an acyclic labeling of $D$.
Now we detour the arcs from $N^-_D(w_q)$ to $w_q$ so that their heads change from $w_q$ to $w_p$.
In this way, we obtain a new acyclic digraph $D^*$.
Since $N^-_D(w_p)$ and $N^-_D(w_q)$ are included in the same clique in $G$, $C(D^*) = C(D)$.
However, $w_q$ is a vertex of indegree $0$ in $D^*$, so the number of vertices of indegree $0$ in $D^*$ is greater than that of vertices of indegree $0$ in $D$, which contradicts the choice of $D$. Therefore $r=\theta_e(G)$.
\end{proof}

%=======================================================

%We have shown that $k(G) = \theta_e(G) - |V(G)| + \widetilde{k}(G)$ in Proposition~\ref{prop:cne} for a graph $G$ having a competition-effective edge clique cover.
%Since a diamond-free graph with at least one edge has a competition-effective edge clique cover, it satisfies the equality.
%Furthermore, if a graph $G$ with at least one edge is diamond-free and planar, then $\theta_e(G)$ can be represented in terms of $|E(G)|$, $c_3(G)$, and $c_4(G)$ as shown in the following proposition.

%%%%%%%%%%%%%%%%%%%%%%%%%%%%%%%%%%
\section{Competition numbers of planar graphs}
%%%%%%%%%%%%%%%%%%%%%%%%%%%%%%%%%%

In this section, we compute competition numbers of some planar graphs in terms of co-competition numbers.

A plane embedding of a planar graph does not change the parameters dealt with in this paper. In this context, we mean by a planar graph a plane embedding of it.
For example, the number of faces in a planar graph means the number of faces in one of its plane embeddings.

We denote the number of faces in a planar graph $G$ by $f(G)$.
\begin{Prop}\label{prop:planar1}
Suppose that $G$ is a nontrivial connected planar graph.
Then $$k(G) \le f(G)$$ and the equality holds if $G$ is triangle-free.
\end{Prop}
\begin{proof}
By the Euler formula $|V(G)| - |E(G)| + f(G) = 2$ for a connected planar graph $G$, it is sufficient to show that
\begin{equation}\label{eqn:planar1}
k(G) \le |E(G)| - |V(G)| + 2
\end{equation}
for every connected planar graph $G$ in order to prove the inequality.
We show (\ref{eqn:planar1}) by induction on the number of edges.
If $G$ has only one edge, then $|E(G)|=1$, $|V(G)|=2$, $ k(G)=1$ and so (\ref{eqn:planar1}) holds.
Suppose that (\ref{eqn:planar1}) holds for all connected planar graphs with $m$ edges.
Let $G$ be a connected planar graph with $m+1$ edges.
Suppose that $G$ is a tree. Then $|V(G)| = |E(G)|+1$.
Since the competition number of a tree is known to be at most one, $k(G) \le 1$, and so  (\ref{eqn:planar1}) holds.
Suppose that $G$ is not a tree.
Then $G$ has an edge $e$ such that $G-e$ is connected.
Since $G$ is planar, $G-e$ is also planar.
Thus, by the induction hypothesis,
\begin{equation}\label{eqn:planar2}
k(G-e) \le |E(G-e)| - |V(G-e)| + 2.
\end{equation}
Clearly
\begin{equation}\label{eqn:addplanar1}
|E(G-e)|=|E(G)|-1 \quad \text{ and }  \quad |V(G-e)|=|V(G)|.
\end{equation}
In addition,
\begin{equation}\label{eqn:addplanar2}
k(G-e) + 1 \ge k(G).
\end{equation}
For, we may add one additional vertex and the arcs from the ends of $e$ to that vertex to an acyclic digraph whose competition graph is $G-e$ together with $k(G-e)$ isolated vertices to obtain an acyclic digraph whose competition graph is $G$ together with $k(G-e) + 1$ isolated vertices.
By \eqref{eqn:planar2}, \eqref{eqn:addplanar1}, and \eqref{eqn:addplanar2}
\[
k(G) \le k(G-e)+1 \le (|E(G)|-1) - |V(G)| + 2 + 1,
\]
which is simplified to $k(G) \le |E(G)| - |V(G)| + 2$.
Therefore (\ref{eqn:planar1}) holds for every connected planar graph $G$.

It is well-known that if $G$ is triangle-free, then  $k(G)=|E(G)|-|V(G)|+2$.
Therefore, if a connected planar graph $G$ is triangle-free, then $k(G) = f(G)$ since $f(G) = |E(G)|-|V(G)|+2$.
\end{proof}

\begin{Cor}\label{cor:planar}
If $G$ is a planar graph, $k(G) \le f(G)$.
\end{Cor}
\begin{proof}
Let $G$ be a planar graph.
If $G$ is an empty graph, then $k(G)=0$ and $f(G)=1$, and so the inequality immediately holds.
Thus we may assume that $G$ has at least one non-isolated vertex.
Let $G^*$ be the subgraph of $G$ induced by the non-isolated vertices.
Let $G_1, \ldots, G_r$ be the components of $G^*$.
Since $G^*$ has no isolated vertex, $k(G_i) \ge 1$ for any $i=1,\ldots,r$.
Then $f(G^*) = \sum_{i=1}^r f(G_i) -r+1$ since the outer face is counted once whenever $f(G_i)$ is computed for $i=1,\ldots,r$.
By Proposition~\ref{prop:planar1}, $k(G_i) \le f(G_i)$ for each $i=1,\ldots,r$ and so $\sum_{i=1}^r k(G_i) -r+1 \le f(G^*)$.

Let $D_i$ be an acyclic digraph whose competition graph is $G_i$ together with $k(G_i)$ isolated vertices for each $i=1,\ldots,r$.
Fix $i \in \{1,\ldots,r\}$.
Since $D_i$ is acyclic, it has a vertex of indegree $0$.
We take an isolated vertex added to obtain $C(D_i)$ and a vertex of indegree $0$ in $D_i$ and denote them by $a_i$ and $u_i$, respectively.
Now we patch $D_1, \ldots, D_r$ by merging (or identifying) $a_{i+1}$ and $u_{i}$ for $i=1,\ldots,r-1$ to obtain a digraph $D$.
By construction, it is obvious that $D$ is acyclic and its competition graph is $G^*$ together with $\sum_{i=1}^r k(G_i) -r+1$ isolated vertices.
Therefore $k(G^*) \le \sum_{i=1}^r k(G_i) -r+1$ and so $k(G^*) \le f(G^*)$.
Since $f(G) = f(G^*)$ and $k(G) \le k(G^*)$, we have $k(G) \le f(G)$.
\end{proof}

A \emph{hole} of a graph is an induced cycle of length at least four.
In 2005, Kim~\cite{kim2005graphs} conjectured that every graph $G$ with $h(G)$ holes satisfies $k(G) \le h(G)+1$, which was proven by Mckay {\it et al.}~\cite{mckay2014competition} in 2014.
By the way, Corollary~\ref{cor:planar} sometimes gives a better bound for planar graphs as one can see from the example given in Figure~\ref{fig:hole}.

\begin{figure}
\begin{center}
\includegraphics[width=5cm]{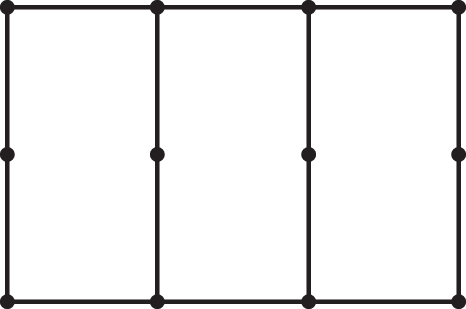}
\end{center}
 \caption{A planar graph $G$ with $f(G)=4$ and $h(G)=6$}
 \label{fig:hole}
\end{figure}

Let $G$ be a planar graph.
Kuratowski's theorem tells us that $G$ contains no $K_5$ as a subgraph, so any maximal clique in $G$ consists of at most four vertices.
For $i=2,3,4$, we denote by $c_i(G)$ the number of maximal cliques of size $i$ in $G$.

We have shown that $k(G) = \theta_e(G) - |V(G)| + \widetilde{k}(G)$ in Proposition~\ref{prop:formula} for a graph $G$ having a competition-effective edge clique cover.
Since a nonempty diamond-free graph has a competition-effective edge clique cover, it satisfies the equality.
Furthermore, if a nonempty graph $G$ is diamond-free and planar, then $\theta_e(G)$ can be represented in terms of $|E(G)|$, $c_3(G)$, and $c_4(G)$ as shown in the following proposition.

\begin{Prop}\label{prop:planar}
Let $G$ be a connected diamond-free planar graph.
Then
\begin{equation}\label{eqn:formula}
k(G) = f(G) + \widetilde{k}(G) - 2c_3(G) - 5c_4(G) - 2.
\end{equation}
\end{Prop}
\begin{proof}
If $G$ has only one vertex, then $k(G)=0$, $f(G)=1$, $\widetilde{k}(G)=1$, $c_3(G)=c_4(G)=0$ and so (\ref{eqn:formula}) holds.

Suppose that $G$ has at least two vertices.
Since $G$ is connected, $G$ has at least one edge.
%Since $G$ is connected and planar, by Euler's formula,
%\begin{equation}\label{eqn:h1}
%|V(G)|-|E(G)|+f(G)=2.
%\begin{equation}
Let $\CCC$ be a minimum edge clique cover of $G$ consisting of maximal cliques.
Since $G$ is diamond-free, each edge of $G$ belongs to exactly one maximal clique.
Thus $\CCC$ consists of all the maximal cliques of $G$.
On the other hand, since $G$ is connected and planar, any maximal clique of $G$ has size $2$ or $3$ or $4$.
Thus
\begin{equation}\label{eqn:h0}
\theta_e(G) = c_2(G) + c_3(G) + c_4(G).
\end{equation}
Since $G$ is diamond-free, the maximal cliques of $G$ are mutually edge-disjoint.
%In addition, any maximal clique of size $i$ has ${i \choose 2}$ edges.
Therefore $|E(G)|={2 \choose 2}c_2(G) + {3 \choose 2}c_3(G) + {4 \choose 2}c_4(G)$ and so $c_2(G) = |E(G)| - 3c_3(G) - 6c_4(G)$.
By substituting this into (\ref{eqn:h0}), we have
\begin{equation}\label{eqn:h1}
\theta_e(G) = |E(G)| - 2c_3(G) - 5c_4(G).
\end{equation}
Since $G$ is a nonempty diamond-free graph, it has a competition-effective edge clique cover by Corollary~\ref{cor:diamond2} and so
$k(G) = \theta_e(G) - |V(G)| + \widetilde{k}(G)$
by Proposition~\ref{prop:cne}.
This equality together with (\ref{eqn:h1}) gives $k(G) = |E(G)| -|V(G)| - 2c_3(G) - 5c_4(G) + \widetilde{k}(G)$.
Since $G$ is connected and planar, by Euler's formula, (\ref{eqn:formula}) follows.
\end{proof}
\noindent
Now we have a result giving bounds for competition numbers of connected diamond-free planar graphs in terms of the number of faces, the number of maximal cliques of size $3$, and the number of maximal cliques of size $4$, which are easily computed.

\begin{Thm}
Let $G$ be a nontrivial connected diamond-free planar graph. Then
\[
\max \{ f(G) - 2c_3(G)-5c_4(G), 1 \} \le k(G) \le f(G) - c_3(G) - 3c_4(G).
\]
\end{Thm}
\begin{proof}
Since $G$ is diamond-free, the maximal cliques of $G$ are mutually edge-disjoint.
Therefore, in $G$, we may delete an edge in each maximal clique of size $3$ and three edges in each maximal clique of size $4$ so that any maximal clique of size $3$ becomes an induced path of length $2$ and any maximal clique of size $4$ becomes an induced path of length $3$.
In this way, we deleted exactly $c_3(G)+3c_4(G)$ edges to have the resulting graph, say $G'$, connected.
Since $G$ is planar, $G'$ is planr.
Then $f(G') = f(G) - c_3(G)-3c_4(G)$.
Since $G'$ has at least one edge, $k(G') \le f(G')$ by Proposition~\ref{prop:planar1}.

Let $D'$ be an acyclic digraph whose competition graph is $G'$ together with $k(G')$ isolated vertices.
Let $\ell$ be an acyclic labeling of $D'$.
Take a maximal clique $K$ of size $3$ or $4$ in $G$.
By the choice of $D'$, each pair of vertices in $K$ that are adjacent in $G'$ has a common out-neighbor in $D'$.
Now we take such common out-neighbors.
Among these common out-neighbors, let $z$ be the vertex with the smallest $\ell$-value.
We add to $D'$ the arcs from the vertices in $K$ to $z$ without creating multiple arcs.
We continue to add arcs in this way for the remaining maximal cliques of size $3$ or $4$ in $G$.
We denote the resulting digraph by $D$.
Since each of the arcs added goes from a vertex with a higher $\ell$-value to a vertex with a lower $\ell$-value, $\ell$ is an acyclic labeling of $D$ and so $D$ is acyclic.
Clearly, $E(G)\subset E(C(D))$.
Since the way of deleting the edges to obtain $G'$ does not leave any triangle, every maximal clique in $G'$ is of size $2$ and therefore the vertex chosen in $D'$ for each maximal clique because it has the smallest $\ell$-value has exactly two in-neighbors in $D'$.
Thus $E(C(D))\subset E(G)$.
Hence the competition graph of $D$ is $G$ together with $k(G')$ isolated vertices.
Therefore $k(G) \le k(G')$ and the upper bound is obtained.

To obtain the lower bound, we recall  Proposition~\ref{prop:base2} and Proposition~\ref{prop:planar}, which give the inequality $k(G) \ge f(G) - 2c_3(G)-5c_4(G)$.
Since $G$ is a nonempty connected graph, $k(G) \ge 1$ and so we obtain the desired lower bound.
\end{proof}

\section{Concluding remarks}
%%%%%%%%%%%%%%%%%%%%%%%%%%%%%%%%%%%%%%%%%%%%%%%%%%%%%%%%%%%%%%%%%%%%%%%%%%%%%%%%%%%
We have a strong belief that every nonempty graph has a competition-effective edge clique cover.
If this conjecture turns out to be true, then every graph $G$ satisfies the equality $k(G) = \theta_e(G) - |V(G)| + \widetilde{k}(G)$ by Proposition~\ref{prop:cne}.
%=======================================================

%
%%%%%%%%%%%%%%%%%%%%%%%%%%%%%%%%%%%%%%%%%%%%%%%%%%%%%%%%%%%%%%%%%%%%%%%%%%%%%%%%%%%%
%\section{Concluding remarks}
%%%%%%%%%%%%%%%%%%%%%%%%%%%%%%%%%%%%%%%%%%%%%%%%%%%%%%%%%%%%%%%%%%%%%%%%%%%%%%%%%%%%
%
%We conjecture that every graph $G$ satisfies the equality $k(G) = \theta_e(G) - |V(G)| + \widetilde{k}(G)$.
%In addition, we have a strong belief that every graph with at least one edge has a competition-effective edge clique cover, which implies the above conjecture by Proposition~\ref{prop:cne}.

%%%%%%%%%%%%%%%%%%%%%%%%%%%%%%%%%%%%%%%%%%%%%%%%%%%%%%%%%%%%%%%%%%%%%%%%%%%%%%%%%%%
\section{Acknowledgement}
%%%%%%%%%%%%%%%%%%%%%%%%%%%%%%%%%%%%%%%%%%%%%%%%%%%%%%%%%%%%%%%%%%%%%%%%%%%%%%%%%%%

The research of the second author and the third author was supported by the National Research Foundation of Korea (NRF) funded by the Korea government (MEST) (No.\ NRF-2015R1A2A2A01006885, No.\ NRF-2017R1E1A1A03070489) and by the Korea government (MSIP).
The first and second author's research was supported by Basic Science Research Program through the National Research Foundation of Korea(NRF) funded by the Ministry of Education(NRF-2018R1D1A1B07049150).

%%%%%%%%%%%%%%%%%%%%%%%%%%%%%%%%%%%%%%%%%%%%%%%%%%%%%%%%%%%%%%%%%%%%%%%%%%%%%%
%\bibliographystyle{plain}
%\bibliography{bib20180105.bib}

\end{document}